\documentclass[english]{amsart}

\usepackage[T1]{fontenc}
\usepackage{color}
\usepackage{textcomp}
\usepackage{tikz, hyperref, aliascnt, color, mathtools}
\usepackage{enumerate}

\definecolor{shade1}{rgb}{.2, .2, .2}
\definecolor{shade2}{rgb}{.5, .5, .5}
\definecolor{shade3}{rgb}{.8, .8, .8}
\definecolor{shade4}{rgb}{.95, .95, .95}
\usetikzlibrary{decorations.pathreplacing}
\usetikzlibrary{patterns}

\usetikzlibrary{arrows,shapes,positioning}
\usetikzlibrary{decorations.markings}
\tikzstyle arrowstyle=[scale=1]
\tikzstyle directed=[postaction={decorate,decoration={markings,
    mark=at position .65 with {\arrow[arrowstyle]{stealth}}}}]
\tikzstyle reverse directed=[postaction={decorate,decoration={markings,
    mark=at position .65 with {\arrowreversed[arrowstyle]{stealth};}}}]

\usepackage{subfig}
\usepackage{amsmath}
\usepackage{amssymb}
\usepackage{amsthm}
\usepackage[english]{babel}
\usepackage{graphicx}

\newtheorem{theorem}{Theorem}[section]
\newtheorem*{theorem*}{Theorem}
\newtheorem{lemma}[theorem]{Lemma}

\newtheorem{corollary}[theorem]{Corollary}
\theoremstyle{definition}
\newtheorem*{definition}{Definition}\newtheorem*{remark}{Remark}

\newenvironment{proofof}[1]{\noindent {\bf{Proof of #1.}}}{ \hfill\qed\\ }

\newcommand{\Ob}[2]{\Omega^{#1}_{\theta,\overline{#2}}}
\newcommand{\On}[2]{\Omega^{#1}_{\theta,{#2}}}
\newcommand{\tOn}[2]{\tilde \Omega^{#1}_{\theta,{#2}}}
\newcommand{\Rb}[1]{R_{\overline{#1}}}
\newcommand{\Rn}[1]{R_{#1}}

\def\Z{\mathbb{Z}}
\def\R{\mathbb{R}}
\def\A{\mathcal{A}}
\def\G{\mathcal{G}}
\def\az{\A^{\Z^2}}
\def\Tt{T_{\theta}}
\def\S1{\mathbb{S}^1}
\def\i{{i,j}}

\def\Tt{T_{\theta}}
\def\tTt{\tilde T_{\theta}}
\def\A{\mathcal{A}}

\def\bC{{\partial\mathcal C_i}}
\def\hat{\widehat}

\begin{document}
\title[]{Minimality of the Ehrenfest wind-tree model}

\author {Alba M\'alaga Sabogal}
\address{Aix Marseille Universit\'e, CNRS, Centrale Marseille, I2M, UMR
  7373, 13453 Marseille, France}
\curraddr{I2M, CMI, 39 rue Joliot-Curie, F-13453 Marseille Cedex 13\\ France}
 \email{alba.malaga-sabogal@univ-amu.fr}
\email{alba.malaga@polytechnique.edu}

\def\curraddrname{{\itshape Address}}
\author{Serge Troubetzkoy}

  \curraddr{I2M, Luminy\\ Case 907\\ F-13288 Marseille CEDEX 9\\ France}

 \email{serge.troubetzkoy@univ-amu.fr}
 
\begin{abstract}
We consider aperiodic wind-tree models, and show that for a generic (in the sense of Baire) configuration the wind-tree dynamics is minimal in almost all directions,
and has a dense set of periodic points.
 \end{abstract}
 \maketitle
\section{Introduction}
In 1912 Paul et Tatyana Ehrenfest proposed the wind-tree model in order to interpret the ergodic hypothesis of 
Boltzmann \cite{EhEh}.  In the Ehrenfest wind-tree model, a point
particle (the ``wind'') moves freely on the plane and collides with the usual law
of geometric optics with irregularly placed identical square scatterers
(the ``trees'').  Nowadays we would say ``randomly placed'', but the notion of ``randomness'' was not made precise, in fact it would have
been impossible to do so before 
Kolmogorov laid the foundations of probability theory in the 1930s.
The wind-tree model has been intensively studied by physicists, see for example \cite{BiRo}, \cite{DeCoVB}, \cite{Ga}, \cite{HaCo}, \cite{VBHa}, \cite{WoLa} and the references therein.

From the mathematical rigorous point of view,  there have been many recent 
results about the dynamical properties of a periodic version of wind-tree models,
scatterers are identical square obstacles one obstacle centered at each
lattice point. The periodic wind-tree model has been shown to be recurrent  (\cite{HaWe}, \cite{HuLeTr},\cite{AvHu}), to have abnormal diffusion (\cite{DeHuLe},\cite{De}), 
and to have an abscence of egodicity in almost every direction (\cite{FrUl}).
Periodic wind-tree models naturally yield infinite periodic translation surfaces, ergodicity in almost every direction for such surfaces have been obtained only in a few situations
\cite{HoHuWe}, \cite{HuWe}, \cite{RaTr}.

On the other hand for randomly placed obstacles, from the mathematically rigorous point of view, up to know  it has only been shown that if at each point of the lattice $\mathbb{Z}^2$ we either center a square obstacle of fixed
size or omit it in a random way, then
the generic in the sense of Baire wind-tree model is recurrent and has a dense set of periodic points  (\cite{Tr1}).

In this article we continue the study of the Baire generic properties of wind-tree models.  We study a random version of the wind-tree model: the plane is tiled by one by one cells with corners on the lattice  $\Z^2$, in each cell we  place a square tree of a fixed size with  the center
chosen randomly. 
Our main result is that for the generic in the sense of Baire wind-tree model,
for almost all directions the wind-tree model is minimal, in stark contrast to the situation for the periodic wind-tree model which can not have a minimal direction.\footnote{K.~Fr\k{a}czek explained to us
that this follows from arguments close to those in the article \cite{Be}.} This result can be viewed as a topological version of the Ehrenfests question.

The method of proof is by approximation by finite wind-tree models where the dynamics is well understood.
There is a long history of proving results about billiard dynamics by approximation which began with the article of Katok and Zemlyakov \cite{KaZe}.  
This method was used in several of the results on  wind-tree models mentioned above 
\cite{HuLeTr},\cite{AvHu},\cite{Tr1}, see \cite{Tr} for a survey of some other usages in billiards.
The idea of approximating infinite measure systems by compact systems was first studied in \cite{MS}.

{The structure of the article is as follows, in Section \ref{secres} we give formal statements of our results. In Section \ref{secnot} we collect the
notation necessary for our setup.  In Section \ref{secmin}, \ref{secescape} and \ref{secperiod} are devoted to the proof of
different parts of the main theorem.  Our proofs hold in a more general setting than the one described above, for example we can vary the size of the square, or use certain
other polygonal trees.  We discuss such extensions of our result in Section \ref{secgen}. Finally in Appendix \ref{secmindis} we discuss the
relationship between the usual convention on the orbit of singular points for  interval exchange transformations and
for polygonal billiards, these conventions are not the same. Since we are studying minimality in this article a careful
comparison is made, and certain known results are reproved for a class of maps we call eligible. In particular the IETs
arising from billiards with the billiard convention { for orbits arriving at corners of the polygon} are eligible maps.}

\section{Statements of Results}\label{secres}
We consider the plane $\R^2$ tiled by one by one closed square {\em cells} with corners on the lattice  $\Z^2$. Fix $r \in [1/4,1/2)$.
We consider the set of $2r$ by $2r$ squares, with vertical and horizontal sides, centered at $(a,b)$ contained in the unit cell $[0,1]^2$, this set is naturally parametrized by
$$\A := \{t=(a,b): r \le a \le 1-r, \ r \le b \le 1-r\}$$
 with the usual topology inherited from $\R^2$.
Our parameter space is $\az$ with the product topology. It is a Baire space.
Each parameter $g= (a_{\i},b_{\i})_{(\i) \in \Z^2} \in \az$ corresponds to a wind-tree table in the plane in the following manner:
the tree inside the cell corresponding to the lattice point $(\i) \in \Z^2$  is a $2r$ by $2r$ square with center at position $(a_\i,b_\i) + (\i)$. 
The wind-tree table $B^g$ is the plane $\R^2$ with the interiors of the union of these trees removed. Note that trees can intersect only at the boundary of cells.

Fix a direction $\theta \in \S1$. The billiard
flow in the direction $\theta$ is the free motion on the interior of $B^g$ with elastic collision from the boundary of $B^g$ (the boundary of the union of the trees).
The \emph{billiard map}  $\Tt^g$ in the direction $\theta$ on the table is the first return to the boundary. If the flow orbit arrives at a 
corner of the table, the collision is not well defined, and we choose not to define the billiard map, i.e.\ the orbit stops at the last collision with the boundary before reaching the corner;  also backwards orbits starting at a corner of a tree are not defined, but forward orbits starting at a corner are defined. 
Once launched in the direction $\theta$, the billiard direction can only achieve four directions $\{\pm \theta, \pm (\theta - \pi)\}$;
 thus the phase space $\Omega^g_\theta$ of the billiard map $\Tt^g$ is a subset of  the cartesian product of the boundary with  these four  directions. It contains precisely the pairs $(s,\phi)$  such that at $s$ the direction $\phi$ 
 points to the interior of the table, i.e.\ away from the trees.
 The billiard map will be called {\em minimal} if the orbit or every point is dense.

The set of periodic points is called {\em locally dense} if there exists a $G_\delta$-subset of the boundary which is of full measure, such that for every $s$ in this set, there is a dense set of inner-pointing directions $\theta \in \S1$ for which $(s,\theta)$ is periodic. We call a forward (resp.\ backward) $\Tt^g$-orbit a \emph{forward (resp.\ backward) escape orbit} if it visits any compact set only a finite number of times.

\begin{theorem}\label{t1}
There is a dense $G_{\delta}$  set of parameters $\G$ such that for each $g \in \G$:
\begin{enumerate}[i)]
\item for a dense-$G_\delta$ set of full measure of $\theta$  the billiard map $\Tt^g$ is minimal and has forward and backward escape orbits,
\item the map $T^g$ has a dense set of periodic points,
\item if $r$ is rational, then the map $T^g$ has a locally dense set of periodic points,
\item no two trees intersect.
\end{enumerate}
\end{theorem}

All the sets mentioned in the theorem depend on the fixed parameter $r$.
From our definition of  minimality we conclude
\begin{corollary}
The backwards orbit of any forward  escape orbit is dense (and vice versa) for each $g \in \G$.
\end{corollary}
\begin{corollary}
The billiard flow on the wind-tree table is also minimal for each $g \in \G$.
\end{corollary}

We would like to point out that there is an old theorem of Gottshalk that (a stronger version of) minimality is impossible in locally compact spaces \cite[Theorem B]{Go}; more precisely for a homeomorphism of a locally compact metric space $X$, if the forward orbit of every point $y \in Y$ is dense in $Y$, then $Y$ is compact.
This result does not apply directly to our situation:  our map is not a homeomorphism, 
the dynamics is not defined everywhere, and where it is not defined it is discontinuous.  There is a standard way of changing the topology to make the map a homeomorphism (this construction is well described in the context of
interval exchanges in \cite[Section 2.1.2]{MMY}). For any wind tree table 
$g \in \az$, including the periodic ones the  topology obtained from $\Omega^g_\theta$ will be locally compact, thus Gottshalks result apply, the wind tree model can never be forward minimal. In fact, in Theorem \ref{t1}.i) we construct examples of escape orbits.

\section{Notations and preparatory remarks}\label{secnot}

{As already mentioned in previous section the billiard map $\Tt^g$ is not defined at a point whose next collision is with a corner, and the inverse billiard map $(\Tt^g)^{-1}$ is not defined at a corner.
In the world of billiards or (flat surfaces), a saddle connection is a flow-orbit going from a corner of a tree to some corner (maybe the same one). Because of the above convention, for the map there is a saddle connection  starting at a point $x$ if, for some $k\ge 0$, $T^k(x)$ is defined but $T^{k+1}(x)$ and $T^{-1}(x)$ are not defined;  then the {\em saddle connection} is the orbit $$\{x,Tx,T^2(x),\dots,T^k(x)\}.$$}
 
 A direction $\theta$ is called \emph{exceptional} if there exists a saddle connection for $\Tt^g$. As there are countably many corners, there are at most countable many saddle connections and thus at most countably many exceptional directions.
 
 For any positive integer $N$,  we define $\Rb{N}$ to be the closed rhombus (square) $\{(x,y):|x|+|y|\le N+\frac12\}$ and we define then $\Ob{g}{N}$ to be  $ \Omega^g_\theta \cap (\Rb{N}  \times \{ \pm \theta, \pm(\theta-\pi)\})$. 
 Let $E_N$ the set of pairs $(\i)$ so that the interior of the $(\i)$-th cell is contained in $\Rb{N}$, and let $\Rn{N}$ be the interior of the union of the closed cells indexed by $E_N$.
Let us also define  $\On{g}{N}$ to be $ \Omega^g_\theta \cap (R_N \times \{ \pm \theta, \pm(\theta-\pi) \})$.

Suppose that $N$ is an integer satisfying $N \ge 2$. We will call a parameter  \emph{$N$-tactful} if for each cell inside the rhombus $\Rb{N}$, the corresponding tree
is contained in the interior of its cell. We will call an $N$-tactful parameter \emph{$N$-ringed}, if
the boundary of $\Rb{N}$ is completely covered by trees.
We call a parameter  \emph{tactful} if it is \emph{$N$-tactful} for all $N$.

For $N$-ringed parameters there is a compact connected rational billiard table $\Rb{N} \cap B^f$, called the \emph{$N$-ringed table},
contained in the rhombus $\Rb{N}$  (see Figure \ref{fig1}). 
The  corresponding phase space is $\Ob{f}{N}.$ It contains $\On{f}{N}$ which is compact for any $N$-tactful parameter. 
A direction $\theta$ is called {\em $(f,N)$-exceptional} if there is a saddle connection inside $\Ob{f}{N}$. 

There are at most countably many exceptional directions, and for all non-exceptional directions, $\Ob{f}{N}$ is a  minimal set for the billiard map  $\Tt^f$. {(We reprove this result in our context in Corollary \ref{c9} of the Appendix.)}

\begin{figure}[t]
\begin{minipage}[ht]{0.5\linewidth}
\centering
\begin{tikzpicture}[scale=0.75]
\foreach \i in {-1.5,-1.0,...,0.5}
{\draw[thin] (\i-1,\i-0.5+2) rectangle +(0.5,0.5);
\draw[thin] (\i+1.5,\i-1.0) rectangle +(0.5,0.5);
\draw[thin] (\i-1.0,-\i-2) rectangle +(0.5,0.5);
\draw[thin] (\i+1.5,-\i+0.5) rectangle +(0.5,0.5);}
 
 \draw[thin] (0.1,0.2) rectangle +(0.5,0.5);
 \draw[thin] (0.3,-0.7) rectangle +(0.5,0.5);

 \draw[thin] (-0.8,-0.6) rectangle +(0.5,0.5);
\draw[thin] (-0.7,0.4) rectangle +(0.5,0.5);

 \foreach \i in {-3,-2,...,2}
 \foreach \j in {-3,-2,...,2}
\draw[dotted](\i,\j) rectangle +(1,1);

\end{tikzpicture}
\caption{A 2-ringed configuration.}\label{fig1}
\end{minipage}\nolinebreak
\begin{minipage}[ht]{0.5\linewidth}
\centering
\begin{tikzpicture}[scale=0.75]
\foreach \i in {-1.5,-0.5,0.5}
{\draw[thin] (-0.05 + \i-1,\i-0.5+2+ 0.07) rectangle +(0.5,0.5);
\draw[thin] (0.06 + \i+1.5,\i-1.0 -0.06) rectangle +(0.5,0.5);
\draw[thin] (-0.07 + \i-1.0,-0.03 -\i-2) rectangle +(0.5,0.5);
\draw[thin]  (  0.05+ \i+1.5,    -\i+0.5+0.05 ) rectangle  +(0.5,0.5);}

\foreach \i in {-1.0,0}
{\draw[thin] (0.05 + \i-1,\i-0.5+2 - 0.05 ) rectangle +(0.5,0.5);
\draw[thin] (-0.04 + \i+1.5,0.05 + \i-1.0) rectangle +(0.5,0.5);
\draw[thin] (+0.05 +\i-1.0,-\i-2 +0.05 ) rectangle +(0.5,0.5);
\draw[thin]  ( - 0.05+ \i+1.5,    -\i+0.5- 0.05 ) rectangle  +(0.5,0.5);}
 
 \draw[thin] (0.1,0.2) rectangle +(0.5,0.5);
 \draw[thin] (0.3,-0.7) rectangle +(0.5,0.5);

 \draw[thin] (-0.8,-0.6) rectangle +(0.5,0.5);
\draw[thin] (-0.7,0.4) rectangle +(0.5,0.5);

 \foreach \i in {-3,-2,...,2}
 \foreach \j in {-3,-2,...,2}
\draw[dotted](\i,\j) rectangle +(1,1);

\end{tikzpicture}
\caption{A small perturbation.}\label{fig2}
\end{minipage}\nolinebreak

\end{figure}

We need to describe  $\On{g}{N}$  more concretely for any $N$-tactful $g$.
 Note
that if the tree is contained in the interior of a cell then if $s$ is a corner of this tree there are three directions pointing to the interior of the table, 
while for all other $s$ there are only two such directions  (see Figure \ref{fig3}). Intersecting trees can have slightly different behavior, but we do not need to describe it since they will not occur in our proof.

 We think of  the contribution of each tree to $\On{g}{N}$ as the union of four closed intervals indexed by $\phi \in \{  \pm \theta, \pm(\theta-\pi) \}$, each of these
intervals corresponds to the cartesian product of  the two intersecting sides  of the tree with a fixed inner pointing direction  $\phi$ (see Figure \ref{fig3}) (as before the word inner means pointing into the table, so
away from the tree). In the proof we will think of each of these intervals as  $I=[0,2\sqrt2r]$ since it corresponds to the diagonal (of length $2\sqrt2r$) of the tree centered at $(a,b)$. Note that the billiard map in a fixed direction has a natural invariant measure. Let $p$ be the arc-length parameter on $I$, then $I=I_1\cup I_2$ where each $I_i$ is an interval and the invariant measure is of the form $K_i\operatorname{d}p$ on $I_i$ with $K_i$ an explicit constant. We stick to the use of $I$ in order to avoid manipulating the constants $K_i$ (which are direction dependent) all the time.

To make our map orientation preserving we choose the orientation
of these intervals in the following way;  use the clockwise orientation inherited from the tree for $\phi \in \{\theta, \theta - \pi\}$ and the counterclockwise orientation of the other two values of $\phi$.
In particular this parametrization does not depend on the angle $\phi$. (See figures \ref{fig3},\ref{fig4}). Despite that fact that these ``intervals'' come naturally as subsets of $\R^2$, we will think of $\On{g}{N}$ and $\Omega^{g}_{\theta}$ as formal 
disjoint union of one-dimensional intervals.

For any $N$-tactful $g$, let $\mathcal J^g_N$ be the collection of all the intervals as described  
arising from the trees in $\Rn{N}$. Note that the trees straddling the rhombus do not contribute to
this collection.

\begin{figure}[t]

\begin{minipage}[ht]{0.4\linewidth}
\centering
\begin{tikzpicture}[scale=1.25]

\draw[] (0,0) rectangle +(2,2);

\draw [](1,2) edge[->]  (1.3,2.3) ;
\draw [](1,2) edge[->]  (0.7,2.3) ;

\draw [](1,0) edge[->]  (1.3,-0.3) ;
\draw [](1,0) edge[->]  (0.7,-0.3) ;

\draw [](0,1) edge[->]  (-0.3,1.3) ;
\draw [](0,1) edge[->]  (-0.3,0.7) ;

\draw [](2,1) edge[->]  (2.3,1.3) ;
\draw [](2,1) edge[->]  (2.3,0.7) ;

\draw [](2,2) edge[->]  (1.7,2.3) ;
\draw [](2,2) edge[->]  (2.3,2.3) ;
\draw [](2,2) edge[->]  (2.3,1.7) ;

\draw [](0,2) edge[->]  (-0.3,2.3) ;
\draw [](0,2) edge[->]  (0.3,2.3) ;
\draw [](0,2) edge[->]  (-0.3,1.7) ;

\draw [](0,0) edge[->]  (-0.3,0.3) ;
\draw [](0,0) edge[->]  (0.3,-0.3) ;
\draw [](0,0) edge[->]  (-0.3,-0.3) ;

\draw [](2,0) edge[->]  (1.7,-0.3) ;
\draw [](2,0) edge[->]  (2.3,0.3) ;
\draw [](2,0) edge[->]  (2.3,-0.3) ;

\end{tikzpicture}
\caption{The phase space of one tree.}\label{fig3}
\end{minipage}\nolinebreak
\begin{minipage}[ht]{0.6\linewidth}
\centering
\begin{tikzpicture}[scale=0.75]

\draw[red,directed] (0,2.3) -- (0,4.3);
\draw[red,directed] (0,4.3) -- (2,4.3);
\draw [red](1,4.3) edge[->]  (0.7,4.6) ;
\draw [red](0,3.3) edge[->]  (-0.3,3.6) ;

\draw[blue,directed] (4.3,2) -- (4.3,0);
\draw[blue,directed] (4.3,0) -- (2.3,0);
\draw [blue](3.3,0) edge[->]  (3.6,-0.3) ;
\draw [blue](4.3,1) edge[->]  (4.6,0.7) ;

\draw[brown,directed] (0,2) -- (0,0);
\draw[brown,directed] (0,0) -- (2,0);
\draw [brown](1,0) edge[->]  (0.7,-0.3) ;
\draw [brown](0,1) edge[->]  (-0.3,0.7) ;

\draw[directed] (4.3,2.3) -- (4.3,4.3);
\draw[directed] (4.3,4.3) -- (2.3,4.3);
\draw [](3.3,4.3) edge[->]  (3.6,4.6) ;
\draw [](4.3,3.3) edge[->]  (4.6,3.6) ;

\end{tikzpicture}
\caption{The phase space is the disjoint union of four closed oriented ``intervals''.}\label{fig4}
\end{minipage}

\end{figure}

 For each tree $t\in \A$ let $U(t,\varepsilon)$ be the the standard $\varepsilon$-neighborhood in $\R^2$ intersected with the interior of $\A$ in $\R^2$.  For any parameter $g=(t_\i)\in\A^{\Z^2}$, consider the open cylinder set $U_N(g,\varepsilon)=\prod_{(\i)\in E_N} U(t_\i,\varepsilon)$.

\section{Proof of Minimality in Theorem \ref{t1}}\label{secmin}

\begin{proof}
{  
In the Appendix we study a class of maps called eligible maps.
The proof of minimality in the theorem is based on 
Lemma \ref{l4} of the Appendix which gives a necessary and sufficient condition for the minimality of an eligible map.
We start by some remarks on the applicability of this lemma, first of all }
 note that in the proof we will only need
to apply  Lemma \ref{l4}
to maps which are tactful. 
However these maps are not eligible.
$\Omega^g_\theta$ is a disjoint union of closed intervals, if we restrict the billiard map to the interior of these intervals it becomes an eligible map
and we can apply the lemma. More precisely we apply Lemma  \ref{l4}
 to the
map $\tTt^g$ which is the map $\Tt^g$ restricted to  $\Omega^g_\theta$ with endpoints of each
interval in $\mathcal J^g_N$ removed.  We call this union of open intervals 
$\tilde \Omega^g_\theta$.   Note that $\Tt^g$ being minimal is equivalent to 
$\tTt^g$ being minimal since the $\Tt^g$-orbit of any corner $x$ is the union of $\{x\}$ with the $\tTt^g$-orbit of $\Tt^g(x)$.

By Lemma \ref{l4} for any tactful $g$, the map $\tTt^g$ being minimal is equivalent  to the statement: for any interval $I \subset \tilde \Omega^g_\theta$  we have
$\bigcup_{k\in\Z} {(\tTt^g)}^k(I)$ covers the whole space 
$\tilde \Omega^g_\theta$. It is enough to show that this happens for a finite union of iterates of $I$.  More precisely it is enough to show that we have sets 
$C_n \subset \tilde \Omega^g_\theta$ satisfying $C_n \subset C_{n+1}$ and 
$\cup_{n \ge 1} C_n = \tilde \Omega^g_\theta$ , such that
$$\forall n \ge 1 \, \forall I  \subset \tilde \Omega^g_\theta \, \exists K,L \text{ s.t }  \bigcup_{k = K}^L {(\tTt^g)}^k(I) \supset C_n.$$
Furthermore it suffices to show this for a countable basis of intervals.

By Corollary \ref{c7} in the Appendix for any $N$-ringed $f$, any $(f,N)$-non-exceptional direction $\theta$ and any interval $I \subset \tOn{f}{N}$, there exists $K,L$ such that 
 \begin{equation}
 \bigcup_{k =K}^L (\tTt^f)^k(I) \supset \tOn{f}{N}.\label{e1}
 \end{equation} 
 
 Now consider the following perturbation of $f$, the new configuration $g$ is arbitrary in the cells which do not intersect $\Rb{N}$,  each tree $(a_\i,b_\i)$ in $\Rn{N}$ is replaced with a tree
 $(a'_\i,b'_\i)$ which is sufficiently close to $(a_\i,b_\i)$ in such a way that the new parameter is still $N$-tactful and the trees in the cells covering the boundary of $\Rb{N}$ are replaced by close trees in such a way that the configuration is $N+1$-tactful (see Figure \ref{fig2}).

The main idea is if Equation \eqref{e1} holds for an  open interval $I$ in $\tOn{f}{N}$ there exists an $\varepsilon>0$ such that  for all $g\in U_N(f,\varepsilon)$, the  Equation \eqref{e1} still holds for $g$ and $I$, namely
 $$
 \bigcup_{k =K}^L \left(\tTt^g\right)^k(I) \supset \tOn{g}{N}.
 $$
 Here we can write the same interval $I$ since 
there is a natural identification between $\tOn{g}{N}$ and $\tOn{f}{N}$ which will be made explicit in the proof.


We remind that each tree contributes four intervals to the phase space of the wind-tree transformation in a given direction. Each of these intervals is a copy of the interval $[0,2\sqrt2r]$. Since in $\tOn{g}{N}$ only the trees contained in $\Rn{N}$ are contributing, $\mathcal J^g_N$ is a collection of $4\operatorname{card}(E_N)=8N(N-1)$ copies of this interval.

For any $N$-tactful $g$, $\mathcal I^g_{N,\theta}$ will be the union of all open coverings
\begin{equation}\label{e1,5}
\left\{\left(\frac{\sqrt2r}{2^N}(i-1),\frac{\sqrt2r}{2^N}(i+1)\right)\cap(0,2\sqrt2r): i=0,\dots,2^{N+1}\right\}
\end{equation}
of $(0,2\sqrt2r)$, one such open covering for each copy of the interval $[0,2\sqrt2r]$ which appears in $\mathcal J^g_N$. 
Note that we have remove the endpoints of the interval $[0,2\sqrt2r]$
because of the discussion at the beginning of the proof of the applicability of
the results in the Appendix.
Thus $\mathcal I^g_{N,\theta}$ is a finite collection of open intervals in $\tOn{g}{N}$. Note also that $\bigcup\limits_N\mathcal I^g_{N,\theta}$ is a topological basis of $\tilde\Omega^g_\theta$. We will call the endpoints of intervals in $\mathcal I^g_{N,\theta}$ \emph{dyadic points} (the endpoints $0$ and $2\sqrt{2}r$ included). 


We will say that there is a *saddle connection  starting at a point $x\in\Omega^g_\theta$ if for some $k\ge 0$ we have: $\big(\Tt^g\big)^i(x)$ is defined and is not a dyadic point for all $0< i\le k$, and $\big(\Tt^g\big)^{k+1}(x)$ and $\big(\Tt^g\big)^{-1}(x)$, if defined, are dyadic points ;  then the {\em *saddle connection} is the orbit segment 
$$\left\{x,\big(\Tt^g\big)x,\big(\Tt^g\big)^2(x),\dots,\big(\Tt^g\big)^k(x)\right\}.$$
A direction $\theta$ is called {\em $((g,N))$-*exceptional} if there is a *saddle connection inside $\Ob{g}{N}$. There are at most countably many *exceptional directions.  Any non-*exceptional direction is non-exceptional, thus Equation \eqref{e1} still holds for any non-*exceptional direction $\theta$,  for $f$ an $N$-ringed configuration.

To prove that the billiard map $\tTt^g$ in a given direction is minimal, it suffices to show that there exists infinitely many $ N $ such that 
\begin{equation}\label{e2}
\exists K,L\;\forall I \in \mathcal I^g_{N,\theta}\bigcup_{k=K}^L {(\tTt^g)}^k(I)\supset \tOn{g}{N}. 
\end{equation}

For any $N$-tactful $g$, let $\mathcal C^g_N(K,L)$ be the collection of all the connected components of $(\tTt^g)^k(I)\cap \tOn{g}{N}$ where $k$ varies from $K$ to $L$. These intervals are open intervals. Each interval $I'$ in $\mathcal C^g_N(K,L)$ is a connected component of $(\tTt^g)^k(I)$ for some $k$ between $K$ and $L$, and some $I\in\mathcal I^g_{N,\theta}$. Thus, for this $k$, $(\tTt^g)^{-k}(I')$ is an interval in $\tOn{g}{N}$, we will call  $\mathcal D^g_N(K,L)$ the collection of all such intervals.

Note that  $\tOn{g}{N}$ and $\tOn{f}{N}$ for every $N$-ringed parameter $f$ and every $N$-tactful $g$ are formally identical. In particular this is true for any $g$ in the cylinder set $U_{N}(f,\varepsilon)$ (defined at the end of the previous section).

By Baire's theorem the set of configurations which are tactful 
is dense since for each $N$ the set of all $N$-tactful configurations is an open dense set.  
Thus we can consider a countable dense set of parameters which are $N$-tactful for all $N$. By modifying the parameters we can assume that each one is $N$-ringed for a certain $N$ still maintaining the density.  Call
this countable dense set  $\{f_i\}$, with $f_i$ being $N_i$-ringed.
We also assume  $N_{i+1} > N_i$.
Suppose $\varepsilon_i$ are strictly positive. Let
$$\mathcal G:=\bigcap_{m \ge 1}\bigcup_{i \ge m}U_{N_i}(f_i,\varepsilon_i).$$
Clearly $\mathcal G$ is a dense $G_\delta$. We claim that there is a choice for $\varepsilon_i$ such that every parameter in $\mathcal G$ gives rise to a wind tree which is minimal in almost all directions.

Fix $f_i$. We already proved that Equation \eqref{e2} holds for $g=f_i$, $N= N_i$ and $\theta$ any direction which is not $(f_i,N_i)$-exceptional (c.f.\ Equation \eqref{e1}). Let $K_i=K_i(\theta)$, $L_i=L_i(\theta)$ be the two integers given by Equation \eqref{e2}. For sake of simplicity, we will denote $\mathcal C_i:=\mathcal C^{f_i}_{N_i}(K_i,L_i)$ the collection of intervals in the covering in Equation \eqref{e2} and we will denote $\mathcal D_i:=\mathcal D^{f_i}_{N_i}(K_i,L_i)$ the collection  defined in the paragraph after Equation \eqref{e2}. The collection of intervals $\mathcal C_i$ is an open cover of the open set $\tOn{f_i}{N_i}$.  We  denote by $\bC$ the set of endpoints in $\Omega^{f_i}_\theta$ of the intervals in $\mathcal C_i$. 

We describe the set $\bC(\theta)$ exactly. Without loss of generality let suppose $K_i(\theta)<0$ and $L_i(\theta)>0$. First, let us consider the set $\On{f_i}{N_i}\setminus \big(\Tt^{f_i}\big)^{K_i}\left(\On{f_i}{N_i}\right)$ which is just the collection of points $x$ whose forward iterate $\big(\Tt^{f_i}\big)^{k}(x)$ is not defined for some time $k\le-K_i(\theta)$, and similarly consider $\On{f_i}{N_i}\setminus \big(\Tt^{f_i}\big)^{L_i}\left(\On{f_i}{N_i}\right)$ for backward orbits. Second, let us consider the following sub-collection of the forward orbits of corners of trees $\big\{\tTt^k(x) : x\in \partial\mathcal J_{N_i} ,\;0\le k\le L_i\big\}$, and similarly for backward iterates. Third, consider the iterates of the endpoints of $I$ for times between $K_i(\theta)$ and $L_i(\theta)$. Then $\bC(\theta)$ is the restriction of the union of these three collections to $\On{f_i}{N_i}$. For each $\theta$ this is a finite collection of points. Note that the each point in in $\bC(\theta)$ is the endpoint of exactly one interval in $\mathcal{C}_i(\theta)$ because $\theta$ is not $(f_i,N_i)$-*exceptional. As we vary the parameters $(g,\theta)$, clearly $\bC(\theta)$ will change. 
Moreover, if the direction $\theta$ is a non-$(f_i,N_i)$-*exceptional direction, then 
the points in the set change continuously in the following sense:  each point in $\bC(\theta)$ has $(f_i,\theta)$ as a point of continuity. 

Let $\Theta_i$ be the set of all directions $\theta$ who are not $(f_i,N_i)$-*exceptional. This set is of measure one since its complement is countable. For every $\theta\in\Theta$ the points in the collection $\bC(\theta)$ are all distinct. Recall that $\On{f_i}{N_i}$ is a union of  oriented intervals, so
the intersection of $\bC(\theta)$  with any of these intervals 
has a natural total order (which is a strict order). These orders induce a strict partial order in $\bC(\theta)$.

So, for every fixed $\theta\in\Theta_i$ it is possible to choose continuously $\delta_i(\theta)>0$ such that the strict partial order on $\bC(\theta')$ is preserved for all $(g,\theta')$ in the open set $ U_{N_i}(f_i,\delta_i(\theta))\times B(\theta,\delta_i(\theta))$ (where $B$ denotes a ball in $S^1$); and thus, Equation \eqref{e2} holds with the same $K_i(\theta)$ and $L_i(\theta)$ for every such $(g,\theta')$.

Furthermore, let us suppose now that $K_i(\theta)$ and $L_i(\theta)$ are optimal for Equation \eqref{e2} to hold.
{ By this we mean the $-K_i(\theta) + L_i(\theta)$ is minimal, and then if there
are several choices $K_i(\theta)$ is chosen maximal still satisfying $K_i(\theta) < 0$.}
 Then $$\Theta_{K,L,M,i}:=\{\theta\in\Theta_i:K_i(\theta)\ge K , L_i(\theta)\le L,\delta_i(\theta)>\frac1M\}$$ is an open set.  
Since $\Theta_i=\bigcup_{K\le0}\bigcup_{L\ge 0}\bigcup_{M\ge 1}\Theta_{K,L,M,i}$ is an increasing union of $\Theta_{K,L,M,i}$ and is of full measure, there exists $K_i,L_i,M_i$ such that $\hat \Theta_i:=\Theta_{K_i,L_i,M_i,i}$ is an open set  of measure larger than $\textstyle\frac1{N_i}$. 
Equation \eqref{e2} holds with the constants $K_i, L_i$ for every $\theta\in\hat \Theta_i$.
Thus $$\Theta:=\bigcap_{N\ge 1}\bigcup_{\{i\ge N\}}\hat \Theta_i$$
is a $G_\delta$-dense set of full measure. Without loss of generality, we suppose $M_i$ is increasing.  If we choose $\varepsilon_i=\frac{1}{M_i}$ in the definition of $\mathcal G$, then this is a set of minimal directions for all tables $g\in\mathcal G$. 
\end{proof}

\section{Proof of the existence of escape orbits.}\label{secescape}

\begin{proof} We will proof the existence of an escape orbit for the parameter set $\mathcal G$ and direction set $\Theta$ defined in the proof of minimality in Section \ref{secmin}. 
For any  $g \in \mathcal G$ and $\theta\in\Theta$, as shown above $\theta$ will be a minimal direction  of $T =\Tt^g$.

Let consider $f$ an $N$-ringed parameter and $\varepsilon>0$.  Let be $N'>N
$ and $f'$ an $N'$-ringed parameter in $U_N(f,\varepsilon)$. Then, for any $\varepsilon'>0$, the set $U_N(f,\varepsilon)\cap U_{N'}(f',\varepsilon')$ is non-empty, moreover if $\varepsilon'$ is small enough then $U_{N'}(f',\varepsilon')\subset U_{N}(f,\varepsilon)$. Let us now fix such an $\varepsilon'$, and let $\theta$ be a direction which is far for horizontal and vertical in the following sense: $\min(|\tan(\theta)|,|\cot(\theta)|)\ge \frac{\varepsilon'}{2r}$. Then for any $g\in U_{N'}(f',\varepsilon')$ we have that $\On{g}{N'+1} \setminus \On{g}{N'}$ is visited by any orbit starting in $\On{g}{N'}$ before reaching $\Omega^g\setminus\On{g}{N'+1}$ since it makes a collision with one of the squares in the ring. In particular for any point $x\in\On{g}{N}$ its orbit cannot escape to $\Omega^g\setminus\On{g}{N'+1}$ without a collision in the ring of obstacles $\Lambda^g_{\theta,N'}:=\On{g}{N'+1}\setminus\On{g}{N'}$. Similarly in the opposite way, an orbit that is already outside cannot come inside without a collision on the ring of obstacles.

Let $g\in \mathcal G$ and let us consider an approximating sequence $f_{i_j}$ such the $g \in U_{N_{i_j}}(f_{i_j},\varepsilon_{i_j})$ for all $j$. For simplicity of notation let $\Lambda_j=\Lambda^g_{\theta,N{i_j}}$. For any non horizontal and non vertical direction $\theta$, the condition $\min(|\tan(\theta)|,|\cot(\theta)|)\ge \frac{\varepsilon_{i_j}}{2r}$ discussed above  is verified for $j$ large enough since $\varepsilon_i$ is a decreasing sequence going to zero.
Thus, we can choose $J$ so large that for any $j \ge J$, the $\varepsilon_{i_j}$ is sufficiently small  such that for any $x \in \On{g}{N_{i_j}}$
the $T$ orbit of $x$ must visit the set $\Lambda_j$ before reaching $\Lambda_{j+1}$. The same is true in the opposite direction: no orbit can go from $\Omega^g_{\theta}\setminus\Omega^g_{\theta,1+N_{i_{j+1}}}$ to $\Lambda_{j}$ without a collision on $\Lambda_{j+1}$. (A similar statement holds for $T^{-1}$. )

Thus the far away dynamics of $T$ can be understood via the following transformations. 
For any $j>J$ and any $x\in \Lambda_j$, let $S^+(x)$ be $T^k(x)$, the first visit to $\Lambda_{j\pm1}$. 
(Note that this is not a first return map to the union of the $\Lambda_j$ and it is not invertible.)  
Similarly, $S^-(x)$ for any $x$ in $\Lambda_j$ is $T^{-k}(x)$ where $k$ is  minimal such that $T^{-k}(x)\in \Lambda_{j\pm1}$.  
For $x\in \Lambda_j$, let $R^+(x)=T^{k'}(x)$ where $k'$ is the maximal $i\ge0$ so that $x,T(x),\dots,T^i(x)\in \Lambda_j$. 
Similarly we define for $x\in \Lambda_j$, $R^-(x)=T^{-k}(x)$ where $k$ is the maximal $i\ge0$ so that $x,T^{-1}(x),\dots,T^{-i}(x)\in \Lambda_j$. 

Where defined, these transformations satisfy: 
\begin{equation}\label{e676}
\begin{array}{rcccl}
R^-\circ S^+&=&S^+&=&R^-\circ R^+\circ S^+,\\ R^+\circ S^-&=&S^-&=&R^+\circ R^-\circ S^-,
\end{array}
\end{equation}
\vspace{-0.45cm}
$$
\begin{array}{rclcrcl}
S^-\circ S^+&=&R^+&\text{and}&S^+\circ S^-&=&R^-.
\end{array}
$$

Now suppose $\theta\in\Theta$. Note that $\theta$ is not vertical  nor horizontal because these directions are $(f_i,N_i)$-exceptional for every $i$. Consider the compact set $\On{g}{M_j}$ for $j\ge J$.  Let 
$$A_{j,1}  := \left\{x \in \Lambda_j: \ S^{+ }(x)\in\Lambda_{j+1} \right\}.
$$ 
The set $A_{j,1}$ is non-empty since the (forward) $T$-orbit of any corner of a tree in $\Lambda_j$ is dense in $\Omega^g_\theta$, thus it has to get out of $\On{g}{j+1}$ and in doing so it is forced to have a collision in $\Lambda_{j+1}$. Thus the last time this orbit visits $\Lambda_j$ before visiting $\Lambda_{j+1}$ will be an element of $A_{j,1}$. Now inductively define the  set
\begin{eqnarray*}
 A_{j,n+1} := \Big \{ x \in A_{j,n} & : & \exists k > 0, \text{ such that }  \forall i = 1,2,\dots,k-1, \\
 && S^i(x)\not\in\Lambda_j \text{ and } S^kx  \in \Lambda_{j+n}\Big \}.
 \end{eqnarray*}
For each $n \ge 1$ the set $A_{j,n}$ is non-empty by a similar reasoning as above.
Clearly $\overline A_{j,n+1} \subset  \overline A_{j,n}$, thus since $\Lambda_j$ is compact, $B_j := \cap_{n \ge 1}  \overline A_{j,n}$ is non-empty.
We claim that if $x$ is in this intersection, then $x$ is in all of the $A_{j,n}$, and thus the forward
orbit of $x$ never returns to $\Lambda_j$.

Suppose not, then let $m :=  \min \{n \ge 1 : x \in \overline A_{j,n} \setminus A_{j,n}\}$. This implies that 
 for some $k$,  $T^i(x) \not \in \Lambda_{j}$ for $i=1,\dots,k-1$ and $T^k(x)$ is not defined, in fact $T^k(x)$ would
 arrive at a corner of a tree of the obstacle ring $\Lambda_{j+m}$. More precisely chose 
 a sequence $(x_\ell) \subset A_{j,m}$ such that $x_\ell \to x$ and $T^k(x) \in \Lambda_{j+m}$, then $y = \lim_{\ell \to \infty} T^k (x_\ell)$.
 Since our direction $\theta$ is non-exceptional, the forward orbit of $y$ is infinite.  Furthermore,  we have $y \in A_{j+m,n}$ for all $n$, thus $y$ is a forward orbit which never visits $\Lambda_j$ 
and which is backwards singular.  Since $\Lambda_j$ has non-empty interior (it contains intervals), this contradicts the minimality of $T$, thus the forward orbit of $x$ is not singular.

Since $\Tt^g$ is minimal, every point in $B_{j+1}$ must have come from $\Lambda_j$, thus we have $B_{j+1}\subseteq R^+\circ S^+\big(B_j\big)$. This implies $$S^-\circ R^-\big(B_{j+1}\big)\subseteq S^-\circ R^-\circ R^+\circ S^+\big(B_j\big)=S^-\circ S^+\big(B_j\big)=R^+\big(B_j\big)=B_j$$ for all $j \ge J$; { here the first two equalities use
 the relations in Equations \eqref{e676} and the last equality follows from the definitions of $B_j$
 and $R^+$.
 
 Let $C_{j+n} := \big (S^- \circ R^- \big )^n\big(B_{j+n} \big )$. Iterating the above computation
 shows that these sets are nested, $C_{j+n+1} \subset C_{j+n}$ for all $n \ge 0$.
The set $\cap_{n \ge 0} \overline{C}_{j+n}$ is non-empty since
it is contained in the compact set $\Lambda_j$.
The forward orbit of any point in this set is either singular, or an escape orbit.  The argument that
such forward orbits are non-singular is identical to the one given above for the set $B_j$.}

Finally we remark that if $(s,\theta)$ is a forward escape orbit, then $(s,\theta + \pi)$ is a backward escape orbit, and vice versa.
\end{proof}

\section{Proof of density and local density of periodic points}\label{secperiod}
\begin{proof}

The idea behind the proof is similar to what has been done for minimality in Section \ref{secmin}. We first apply a known result to  $N$-ringed parameters.

In this section, the direction $\theta$ varies in the proof, so we abandon the notation $\Tt^g$ and we note the billiard transformation in the wind-tree by $T^g(s,\theta)$.

 For each  point $x=(s,\theta)\in\Omega^g_\theta$ and each $p$ such that $\big(T^g\big)^p(x)$ exists, let us consider the set of directions for which the orbit starting at $s$ hits the same sequence of sides up to time $p$:
 \begin{eqnarray*}
 \Big \{\theta'  : && \hspace{-1.5em} \big(T^g\big)^i(s,\theta')\textnormal{ and }\big(T^g\big)^i(s,\theta)\\
&& \hspace{-1.5em} \textnormal{ lie on the same side of the same tree for }i=1,\dots,p \  \Big \}.
 \end{eqnarray*}
This set is an open interval, we will note by $\theta^g_-(x,p)$, and $\theta^g_+(x,p)$ the lower and the upper bound of this interval. We also consider the interval $(t^g_-,t^g_+)$ where $t_\pm$ is the spatial coordinate of $\big(T^g\big)^p(s,\theta^g_\pm)$.

Fix a $N$-ringed parameter $f$, and $x=(s,\theta)\in \On{f}{N}$ such that $(T^f)^p(x)$ exists. 
Remember that the identification between the phase space is  a formal identity map discussed in Section \ref{secmin}. Since $(T^g)^p$ is locally continuous at $x$, both $\theta^g_-(x,p)$ and $\theta^g_+(x,p)$, and thus $t^g_-(x,p)$ and $t^g_+(x,p)$ vary continuously with respect to $g$ in a sufficiently small neighborhood of $f$. Let $(s^g_*,\theta):=\big(T^g\big)^p(s,\theta)$, then, in a sufficiently small neighborhood of $f$, $s^g_*$ varies continuously with respect to $g$. 

By definition of $t^g_\pm$, for all $s'\in (t^g_-,t^g_+)$, there exists an orbit starting at $(s,\theta')$ and ending at $(s',\theta')$ for some $\theta'\in \big(\theta^g_-(x,p),\theta^g_+(x,p)\big)$. Now, suppose that $x=(s,\theta)$ is $T^f$-periodic of period $p$. Note that $s^f_*=s\in (t^f_-,t^f_+)$. 
So there exists $\theta^g_*(s)$ such that $\big(s,\theta^g_*(s)\big)$ is $T^g$-periodic and its period is a divisor of $p$.

Furthermore we can assume that this neighborhood $V(x)$ of $f$ is so small that $(s,\theta^g_*)$ is $\frac{1}{N}$-close to $(s,\theta)$ (with respect to a fixed usual norm).

We will use the following theorem
\begin{theorem*}\cite[Theorem 1]{BGKT}
In a rational polygon periodic points of the billiard flow are dense in the phase space.
\end{theorem*}
This theorem immediately  implies that  the same is true for the billiard map.
In particular periodic points are dense in $\On{f}{N}$. Let $\{x_1,\dots,x_k\}\subset \On{f}{N}$ be a set of $T^f$-periodic points be such that $\{x_1,\dots,x_k\}$ is $\frac1N$-dense in $\On{f}{N}$.
Combining this with the previous paragraph, we conclude that for every $g$ in the neighborhood $V_N(f)= \bigcap_i V(x_i)$, the set of $T^g$-periodic points is at least $\frac2N$-dense in $\On{g}{N}$.

Let $\{f_i\}\subset \az$ be countable and dense, such that each $f_i$ is $N_i$-ringed for some $N_i$.  Let
$$\mathcal G:=\bigcap_{N\ge 1}\bigcup_{\{i:N_i\ge N\}}V_{N_i}(f_i).$$
Clearly $\mathcal G$ is a dense $G_\delta$. We have shown that every parameter in $\mathcal G$ gives rise to a wind tree with dense periodic points.

Now additionally suppose $r$ is rational. In this case, we can use a stronger property on periodic orbits, it is a part of a special case  of Veech's famous theorem  known as the Veech dichotomy:
\begin{theorem*} \cite[Theorem 1.4]{Ve}\cite[Theorem 5.10]{MaTa} If a polygon $P$ is square tiled then every non-singular orbit 
in an non-exception direction is periodic.
\end{theorem*}

 We will call a parameter $f$ \emph{rationally $N$-ringed} if $f$ is $N$-ringed  and all $a_\i,b_\i$ are rational for all $(\i)\in E_N$. The key  property here is that for any rationally $N$-ringed parameter  the $N$-ringed table is  square-tiled and thus we can apply Veech's theorem:  there exists a countable dense set $\{\theta_j\}\subset\mathbb S^1$ such that every non-singular point of the form $(s,\theta_j)\in\Omega^f_{\theta_j,N}$ is periodic. 
We call such a direction a \emph{periodic direction}.

We assume $\theta_j$ are enumerated so that the maximal combinatorial length of the periodic orbits inside $\Omega^f_{\theta_j,N}$ is increasing with $j$.
Consider the smallest $\ell(f)$ such that $\theta_1, \dots,\theta_{\ell(f)}$ is $\frac1N$-dense. 

Let $f$ be a rationally $N$-ringed parameter. 
Consider a periodic direction $\theta$ and the set $\On{f}{N}$ with saddle connections removed. We decompose this set into its periodic orbit structure; more precisely
this decomposition consists of a finite collection of intervals permuted by the dynamics such that the boundary of each interval from this decomposition is in a saddle connection. We call this collection of intervals $\mathcal D(f,\theta)$. For each $I\in \mathcal D(f,\theta)$, all points in $I$ are periodic of the same period $p$, and we call $\bigcup_{i=0}^{p-1}\big(T^f\big)^i(I)$ a \emph{periodic cylinder}. 

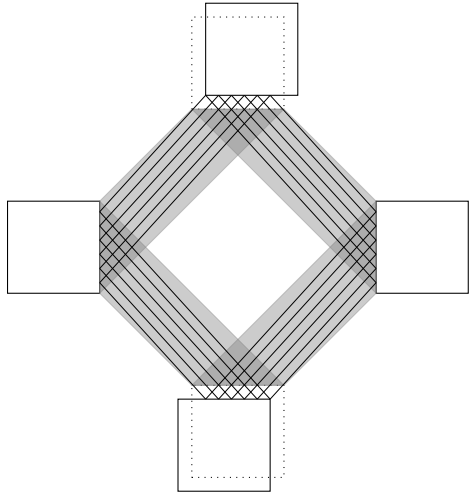
\begin{figure}[tb]
%
%
%
\centering
\begin{tikzpicture}
  \begin{scope}
  [x=1.75cm,y=1.75cm,scale=0.7]
   \draw[dotted] (0,0) rectangle +(1,1);
     \draw[thin] (2,2) rectangle +(1,1);
         \draw[thin] (-2,2) rectangle +(1,1);
           \draw[dotted] (0,4) rectangle +(1,1);
           
    \draw[thin] (0.15,4.15) rectangle +(1,1);
        \draw[thin] (-0.15,-0.15) rectangle +(1,1);
        
  \draw[very thin,fill,color=gray,opacity=0.4] (2,2) -- +(0,1) -- (1,4) -- +(-1,0) -- cycle; 
  \draw[very thin,fill,color=gray,opacity=0.4] (-1,2) -- +(0,1) -- (1,1) -- +(-1,0) -- cycle;
  \draw[very thin,fill,color=gray,opacity=0.4] (0,1) -- +(1,0) -- (2,2) -- +(0,1) -- cycle;
  \draw[very thin,fill,color=gray,opacity=0.4] (0,4) -- +(1,0) -- (-1,2) -- +(0,1) -- cycle;

 \foreach \k in {0,0.2,...,1} \draw[very thin] (2,2.115+0.77*\k)  -- (0.15 + 0.7* \k,4.15); 
\foreach \k in {0,0.2,...,1} \draw[very thin] (-1,2.115+0.77*\k) --   (0.15+ 0.7*\k,0.85); 
 \foreach \k in {0,0.2,...,1} \draw[very thin] (0.15+0.7*\k,0.85)  -- (2,2.885-0.77*\k);
 \foreach \k in {0,0.2,...,1} \draw[very thin] (0.15 + 0.7*\k,4.15)  -- (-1,2.885 -0.77* \k) ;

  \end{scope}
\end{tikzpicture}
\caption{An example of periodic cylinder of length 4 (filled), and  this cylinder after perturbation (striped).}\label{fig5}

\end{figure}

In the general case, we presented a construction that associates to every $T^f$-periodic point $x=(s,\theta)$ and every $g$ in a small enough neighborhood $U_1(s,\theta)$ of $f$, an angle $\theta^g_*(x)$ such that $(s,\theta^g_*(x))$ is $T^g$-periodic. 

Because the periodic points come in cylinders, as described above for $f$, the angles $\theta^g_*(s,\theta)$ and $\theta^g_*(s',\theta)$ will coincide for $s'$ in an open interval  around $s$ (if $g\in U_1(s,\theta)\cap U_1(s',\theta)$).

For each interval $I$ in $\mathcal D(f,\theta)$, we can thus find an interval $I'\subset I$ containing at least $1-\textstyle\frac1{\ell(f)\cdot N}$ proportion of points of $I$ such that the intersection $U_N(f):=\bigcap\limits_{s'\in I'}U_1(s',\theta)$ is open. For all $g\in U_N(f)$ and all $s,s'\in I'$  we have $\theta^g_*(s,\theta)=\theta^g_*(s',\theta)$. 

Furthermore we can assume that this neighborhood $U_N(f)$ of $f$ is so small that $\theta^g_*$ is $\frac{1}{N}$-close to $\theta$ (with respect to a fixed usual norm).

Let $\{f_i\}\subset \az$ be countable dense and such that each $f_i$ is rationally $N_i$-ringed for some $N_i$.  Let
$$\mathcal G:=\bigcap_{N\ge 1}\bigcup_{\{i:N_i\ge N\}}U_{N_i}(f_i).$$
Clearly $\mathcal G$ is a dense $G_\delta$. We claim that every parameter in $\mathcal G$ gives rise to a wind tree with locally dense periodic orbits. For each parameter $g\in \mathcal G$, there exists an infinite subsequence $({f_{i_k}})\subset (f_i)$ such that 
$g\in U_{N_{i_k}}(f_{i_k})$ for all $k$ and $N_{i_k}$ is increasing. For sake of simplicity we denote this subsequence by $(f_k)$. 

Let $m_k$ be the measure of $\bigcup_{I\in \mathcal D(f_k,\theta_j)}I$ (it does not depend on $j$). By definition of $I'$, $\bigcup_{I\in \mathcal D(f_k,\theta_j)}I'$ is of measure at least $\big(1-\textstyle\frac1{\ell(f_k)\cdot N_k}\big)m_k$. Thus $\bigcap_{j=1}^{\ell(f_k)}\bigcup_{I\in \mathcal D(f_k,\theta_j)}I'$ is of measure at least $\big(1-\textstyle\frac1{N_k}\big)m_k$ and thus the complement of the following infinite measure $G_\delta$ set:
$$\bigcap_K\bigcup_{k\ge K}\bigcap_{j=1}^{\ell(f_k)}\bigcup_{I\in \mathcal D(f_k,\theta_j)}I'$$
is of zero measure.\end{proof}

\section{Generalizations}\label{secgen}
Our results hold in a much larger framework.  In the proof of minimality we only used that 
$N$-ringed configurations are dense in the space of all configurations, and that they are rational polygonal billiard tables.
For the local density of periodic orbits we also used that $N$-ringed configurations which are Veech polygonal billiard tables are dense.
Now we give some examples where these properties hold.

1) We stay in the setup discussed in the article but  additionally allow the empty tree denoted by $\emptyset$, thus
the space of parameter is  $\{\emptyset\} \cup  \{(a,b): r \le a \le 1-r, \ r \le b \le 1-r\}.$ The Ehrenfests specifically required that the average distance $A$ between neighboring squares is large compared to $2r$. For any probability distribution $m$ on the continuous part of the space of parameters, if we add a $\delta$ function on the empty tree, then for $c<1$ large enough, the distribution $c\delta+(1-c)m$ verifies almost surely this requirement. However our result tells nothing about a full measure set of parameters for Lebesgue measure. 
 
2) Instead of fixing $r\in [\frac14,\frac12)$, we fix $r$ between $0$ and $\frac12$. If $r\in\big[\frac1{2(n+1)},\textstyle\frac1{2n}\big)$, place at most $n^2$ copies of trees in each cell. We can then form $N$-ringed configuration in a more general sense where we replace the rhombus by an appropriate curve around the origin. One can do so using just $n+1$ copies of the tree in each cell.
 
3) Instead of fixed size squares we use all vertical horizontal squares contained in the unit cell $[0,1]^2$. This set is naturally parametrized by
$$ \{t=(a,b,r): 0 \le a \le 1, 0 \le b \le 1, 0 \le r \le \min(a,b,1-a,1-b)\}$$ where a $2r$ by $2r$ square tree is centered at the point $(a,b)$.
More generally we call a polygon a VH-tree if the sides alternate between vertical and horizontal. For example a VH-tree with 4 sides is a rectangle, with  6 sides is a figure L.  We can use various subsets of VH-trees, for example all VH-trees with at most $2M$ sides ($M \ge 4$ fixed) contained
 in the unit cell.  Or we can use the VH-trees with 12 sides and fixed side length $r \in [1/4,1/3)$ (called $+$ signs). 
 Many other interesting subclasses can be
 considered.
 
 4) Fix a rational triangle $P$, and consider the set of all rescalings of $P$ contained in the unit cell $[0,1]^2$ oriented in such a way that they have either a vertical or horizontal side.
 
 5) One can also change the cell structure to the hexagonal tiling and consider appropriate polygonal trees, for example one can use appropriate classes of
 equilateral triangular trees or hexagonal trees.  

\section{Acknowledgements.}  We thank the anonymous referee whose detailed remarks have greatly improve the article.
We gratefully acknowledge the support of ANR Perturbations and ANR GeoDyM as well as the grant APEX PAD attributed by the region PACA.
This work has been carried out in the framework of the Labex Archim\`ede (ANR-11-LABX-0033) and of the A*MIDEX project (ANR-11-IDEX-0001-02), funded by the ``Investissements d'Avenir'' French Government program managed by the French National Research Agency (ANR)''. 

\appendix
\section{Minimality of discontinuous maps}\label{secmindis}

In this section we develop in a general context the tools we will use for proving minimality of maps with singularities.

\subsection{Definitions}
First, let us make precise the context in which we are using the definition of minimal map. 

\begin{definition}
 Let $X$ be a locally compact metrizable topological space endowed with a Borel measure without atoms. Let $T:X\dashrightarrow X$ be a measure-preserving map. (The dashed arrow stands for the fact that $T$ is possibly not everywhere defined). Let us suppose that  $T$ sends homeomorphically a complement of a discrete set of points to a complement of a (possibly different) discrete set of points. We call such a map \textit{eligible}. 
\end{definition}

\begin{remark}
 If a map $T:X\dashrightarrow X$ is eligible there exists a set $\mathcal S\subset X$ that is discrete and such that $T$ restricted to $X\setminus \mathcal S$ is an homeomorphism. If $\mathcal S_1$ and $\mathcal S_2$ are two such sets, than $\mathcal S_1\bigcap \mathcal S_2$ is also such a set. Indeed, $\mathcal S_1\bigcap \mathcal S_2$ is discrete and $T$ is a homeomorphism on $X\setminus\left(\mathcal S_1\bigcap \mathcal S_2\right)=X\setminus\mathcal S_1\bigcup X\setminus\mathcal S_2$. 
 
 More generally, any intersection of such sets is also such a set. Thus, there exists a minimal discrete set (w.r.t.\ inclusion) such that $T$ is a homeomorphism outside this set.
\end{remark}

\begin{definition}
 Let $T:X\dashrightarrow X$ be an eligible map. We call $\operatorname{Sing}(T)$ the minimal discrete subset of $X$ such that $T$ restricted to $X\setminus \operatorname{Sing}(T)$ is a homeomorphism. We call every point in $\operatorname{Sing}(T)$ a \textit{singularity} of $T$. 
\end{definition}

\begin{remark}
 If $T$ is eligible, then $T^{-1}$ is also eligible and 
\begin{equation}\label{eqT1}\nonumber
 \operatorname{Sing}\left(T^{-1}\right)=X\setminus T\left(X\setminus\operatorname{Sing}(T)\right).
\end{equation}
\end{remark}

 Next, we redefine the notions of images and pre-images of sets in a way that makes clear that we will never apply the eligible transformation on its singular points.
 
\begin{definition}
 Let $T:X\dashrightarrow X$ be an eligible map. For any $A\subset X$, the \textit{image} of $A$ by $T$ is
 \begin{equation}\nonumber
 T(A) :=\{T(x)\;:\;x\in A\setminus\operatorname{Sing}(T)\}.
 \end{equation}
 The \textit{preimage} of $A$ by $T$ is its image by $T^{-1}$, thus:
 \begin{equation}\nonumber
 T^{-1}(A)=\left\{T^{-1}(x)\;:\;x\in A\cap T\left(X\setminus\operatorname{Sing}(T)\right)\right\}.
 \end{equation}
 
 We then define $T^k(A)$ by recurrence for any integer $k$, as follows:  $T^{k+1}(A)=T(T^k(A))$ for any $k\ge 0$; $T^{k-1}(A)=T^{-1}(T^k(A))$ for any $k\le 0$.
\end{definition}

\begin{remark}
The set of singularities, $\operatorname{Sing}(T)$, is closed because it is discrete in a locally compact space. Thus $X\setminus\operatorname{Sing}(T)$ and $T(X\setminus\operatorname{Sing}(T))$ are both open in $X$. It follows that, even with this redefined notion of image and preimage, the image and preimage of an open set are always open. However, we can say nothing about the closedness of the image, or preimage, of a closed set. 
\end{remark}

\begin{definition}
 Let $T:X\dashrightarrow X$ be an eligible map. Let us consider a point $x_0\in X$.
 \begin{itemize}
  \item The \textit{future orbit} of this point is the set of all the positive iterates of $T$ on $x_0$, as long as $T$ is applied to non-singular points. It is noted by $\mathcal O^+(x_0)$, thus:
  \begin{equation}\nonumber
\mathcal O^+(x_0):=\bigcup_{k\ge0}T^k\left(\{x_0\}\right).
  \end{equation}
  \item In a similar way,  the \textit{past orbit} is
  \begin{equation}\nonumber
\mathcal O^-(x_0):=\bigcup_{k\le0}T^k\left(\{x_0\}\right).
  \end{equation}
  \item The \textit{orbit} of a point is the union of its past and future orbit:
  \begin{equation}\nonumber
\mathcal O(x_0):=\bigcup_{k\in\mathbb Z}T^k\left(\{x_0\}\right).   
  \end{equation}
 \end{itemize}
We define the orbit of a set in a similar way. Let $A\subset X$, then
 \begin{itemize}
  \item The \textit{future orbit} of this set is the collection of all the positive iterates of $T$ on $A$. It is noted by $\mathcal O^+(x_0)$, thus:
  \begin{equation}\nonumber
\mathcal O^+(A):=\{T^k\left(A\right): k\ge 0\}.
  \end{equation}
  \item In a similar way,  the \textit{past orbit} is
  \begin{equation}\nonumber
\mathcal O^-(A):=\{T^k\left(A\right): k\le 0\}.
  \end{equation}
  \item The \textit{orbit} of a point is the union of its past and future orbit:
  \begin{equation}\nonumber
\mathcal O^-(A):=\{T^k\left(A\right): k\in\mathbb Z\}.
  \end{equation}
 \end{itemize}
 A \textit{half orbit} is a future or past orbit. 
 
 We say that a point orbit or half orbit is \textit{singular} if it is an orbit or half orbit of a singular point (a point in $\operatorname{Sing}(T)\cup\operatorname{Sing}(T^{-1})$).
 \end{definition}
 
  Remark that the image of a non-empty set may be empty and an orbit may be finite. 

\begin{definition}
 Let $T:X\dashrightarrow X$ be an eligible map. A \textit{connection} is a finite non-periodic orbit. Thus, it is both an orbit of a singularity of $T$ and an orbit of a singularity of $T^{-1}$.
\end{definition}

\begin{remark}
 If $x_0\in X$ is singular for both $T$ and $T^{-1}$, then $\{x_0\}$ is a connection.
\end{remark}

\begin{definition}
 Let $T:X\dashrightarrow X$ be an eligible map. We say that $T$ is \textit{minimal} if and only if every orbit is dense. 
\end{definition}

\begin{remark}
 Let $T:X\dashrightarrow X$ be an eligible map. If $T$ is minimal and has connections, then $X$ is finite. 
\end{remark}

\subsection{Equivalent definition of minimality}

\begin{lemma}\label{l4}
 Let $T:X\dashrightarrow X$ be an eligible map. Then $T$ is minimal if and only if the orbit of every open set covers $X$. 
\end{lemma}
\begin{proof}
Suppose that $T$ is minimal. Let $U$ be an open set, then the orbit of every point meets $U$. More precisely, for every $x\in X$, there exists $k\ge 0$ such that either $x, T^{-1}(x),\dots, T^{-k+1}(x)$ are non-singular for $T^{-1}$ and $T^{-k}(x)\in U$; or $x, T(x),\dots, T^{k-1}(x)$ are non-singular for $T$ and $T^k(x)\in U$. It follows that for every $x\in X$ there exists an integer $k$ such that $x\in T^k(U)$. Thus, $\mathcal O(U)$ is an open covering of $X$.
 
 Reciprocally, let us suppose that the orbit of every open set is an open covering of $X$. Let us consider $x\in X$ and $U$ an open set. Because the orbit of $U$ covers $X$, there exists an integer $k\ge0$ such that $x\in T^k(U)$ or $x\in T^{-k}(U)$. So, there exists $u\in U$ such that $x=T^k(u)$ and $u,T(u),\dots,T^{k-1}(u)$ are non-singular for $T$; or $x=T^{-k}(u)$ and $u,T^{-1}(u),\dots,T^{-k+1}(u)$ are non-singular for $T^{-1}$. Thus one has  $u=T^{-k}(x)$ or $u=T^k(x)$. Thus $T$ is minimal because the orbit of any point intersects any open set.
\end{proof}

\subsection{Keane's minimality criterion}

Keane has shown that interval exchange transformations with no connections are minimal \cite{Ke}.  Keane's proves this fact with the usual
convention that the IET is defined at singular points
via left continuity. This convention does not agree with our convention that billiard orbits
stop when they arrive at a corner.  Thus we do not define IETs at singular
points, this makes an IET an eligible map. Moreover, Keane considered IETs defined on a single interval while in our context the arising IETs are naturally defined on a finite disjoint union of intervals. More  precisely:

\begin{definition}
An IET is an eligible map $T:X\dashrightarrow X$ where $X \subset \mathbb{R}$ is a finite union of open, bounded intervals whose closures are disjoint, and $T$ is a translation on each connected component of $X\setminus\operatorname{Sing}(T)$.  We call $T$ {\em reducible} if a non-trivial finite union of connected components of $X$ is invariant, and otherwise we call $T$ {\em irreducible}.
\end{definition}

Remark:  the billiard map restricted to $\Omega_{\theta}^g $ as described in the article is not an eligible map.
The billiard map  is defined on a disjoint union of closed intervals, if we restrict it to the interior of these intervals it becomes an eligible map.
Furthermore if we restrict to the inside of a ringed it is 
an IET.

Keane's result remains none the less true with a slight adjustment; for completeness we give a proof here. 
\begin{theorem}\label{keane}
An aperiodic, irreducible  IET $T:X\dashrightarrow X$ with no connections is minimal.
\end{theorem}

The proof of this theorem uses the following lemma (compare 
with  \cite[Theorem 1.8]{MaTa}).

\begin{lemma}\label{l-conn}
Suppose that $T:X\dashrightarrow X$ is an  IET with no connections and that
$\mathcal{O}^+(x)$ is an infinite non-periodic forward orbit.  If $I$ is an open interval
with endpoint $x$, then $\mathcal{O}^+(x)$ returns to $I$.
\end{lemma}

\begin{proof}
Since there are a finite number of singularities, there are a finite number 
of trajectories starting at points of $I$ that hit a singularity before crossing 
$I$ again. By shortening $I$ to a subinterval $I'$ with one endpoint $x$
we can assume that no trajectory leaving 
$I'$ hits a singularity before returning to $I'$. Now consider the forward
iterates $T^i(I')$. By the definition of $I'$, these are 
intervals of the same length for each $i \ge 0$ until the interval returns and
overlaps $I'$.  The interval $I'$ must return and overlap $I'$ in a finite time since the total length of $X$ is finite. Let $j$ be the minimum number of iterates needed until $T^i(I')$ overlaps $I'$. $T^j(x) \ne x$ since $x$ is not periodic.

If $T^{j}(x)\in I'$, we are done (see Figure 1 left).  
\begin{figure}
\begin{minipage}[ht]{0.5\linewidth}
\centering
\begin{tikzpicture}
\draw [](-1,0) edge[]  (1,0) ;
\node at (-1,0.25) {\tiny{$y$}} ;
\node at (1,0.25) {\tiny{$x$}} ;
\node at (-1,0) {\tiny{$\circ$}} ;
\node at (1,0) {\tiny{$\circ$}} ;
\draw [](-2.3,-0.2) edge[]  (-0.3,-0.2) ;
\node at (-2.3,-0.45) {\tiny{$T^{j}y$}} ;
\node at (-0.3,-0.45) {\tiny{$T^{j}x$}} ;
\node at (-2.3,-0.2) {\tiny{$\circ$}} ;
\node at (-0.3,-0.2) {\tiny{$\circ$}} ;
\end{tikzpicture}
\end{minipage}\nolinebreak
\begin{minipage}[ht]{0.5\linewidth}
\begin{tikzpicture}
\draw [](-1,0) edge[]  (1,0) ;
\node at (-1,0.25) {\tiny{$y$}} ;
\node at (-1,0) {\tiny{$\circ$}} ;
\node at (-0.3,0.25) {\tiny{$z$}} ;
\node at (-0.3,0) {\tiny{$\circ$}} ;
\node at (1,0.25) {\tiny{$x$}} ;
\node at (1,0) {\tiny{$\circ$}} ;
\draw [](0.3,-0.2) edge[]  (2.3,-0.2) ;
\node at (0.3,-0.45) {\tiny{$T^{j}y$}} ;
\node at (0.3,-0.2)  {\tiny{$\circ$}} ;
\node at (2.3,-0.45) {\tiny{$T^{j}x$}} ;
\node at (2.3,-0.2) {\tiny{$\circ$}} ;
\node at (1,-0.45) {\tiny{$T^{j}z$}} ;
\node at (1,-0.2) {\tiny{$\circ$}} ;
\draw [](-0.5,-0.8) edge[]  (0.8,-0.8) ;
\node at (-0.5,-1) {\tiny{$T^kz$}} ;
\node at (-0.5,-0.8) {\tiny{$\circ$}} ;
\node at (0.8,-1) {\tiny{$T^kx$}} ;
\node at (0.8,-0.8) {\tiny{$\circ$}} ;

\end{tikzpicture}
\end{minipage}
\caption{The two ways $I'$ and $T^{j}(I')$ can overlap}
\end{figure}
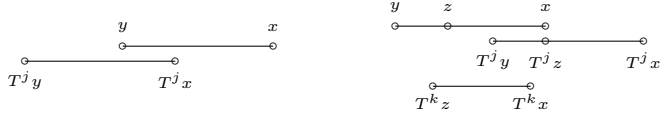
Otherwise (see Figure 1 right) it is the trajectory leaving the other endpoint $y$ of $I'$ which returns to $I'$ at time $j$ and for some $z \in I'$ we have $T^{j}(z) =x$. We now consider the interval $I''$
with endpoints $z$ and $x$ and apply the previous analysis to it. 
Orbiting $I''$ in the forward
direction it must return to $I''$ at a certain (minimal) time $k > j$. We have either $T^k(z)\in I''$ or $T^k(x)\in T^j(I')$ But the first can not happen since
it implies that $T^k(x)\in T^j(I')$, or equivalently  $T^{k-j}(x)\in I'$, which
contradicts the minimality of $k$.
Thus $T^k(x)\in I''$ as required.
\end{proof}

\begin{proofof}{Theorem \ref{keane}} By way of contradiction suppose 
there is a non-periodic  infinite trajectory $\mathcal{O}(x)$ which is not dense.
Let $A \ne X$ be the set of limit points of $\mathcal{O}(x)$. Then $A$ is invariant under the map $T$. Since
$A \ne X$ and $T$ is irreducible one can choose a trajectory $\mathcal{O}(y) \subset \operatorname{int}(X)\cap \overline{A}\setminus\operatorname{int(A)}$ (here the closure and the interior are taken in $\R$). Note that $\mathcal O (y)\subset A$ since $A$ is closed.

We will show that this trajectory is a saddle connection. 
We prove this by contradiction. Suppose it is not true, then $\mathcal{O}(y)$
is infinite in at least one of the two directions. We will
show that this implies that there is an open neighborhood of $y$ contained in $A$, a contradiction to $y$ being a boundary point.
Let $I$ be an open interval with $y$ an endpoint. It is enough to show that there exists an open interval $(y, z) \subset I$ which is contained in $A$. Doing this on both sides will yield our open neighborhood.

Lemma \ref{l-conn} implies that $\mathcal{O}(y)$ hits $I$ again at some point $z$. If
the interval $(y,z) \subset A$ we are done. Suppose not. Then there exists $w \in (y,z)$ which
is not in $A$. Since $A$ is closed, there is a largest open subinterval $I' \subset (y,z)$ containing $w$ which is in the complement of $A$. Let $v$ be the endpoint of $I'$ closest to $y$. Then, since $A$ is closed,
$v \in A$ and the trajectory through $v$ must be a saddle connection. For if it were infinite in
either direction, it would intersect $I'$. Since $A$ is invariant, 
this contradicts
that $I'$ misses $A$.
\end{proofof}

\begin{corollary}\label{c7}
If $T:X\dashrightarrow X$ is an aperiodic and irreducible IET with no connections, then the orbit of every open interval $I$ covers $X$. Moreover, there exists $K,L$ such that $\bigcup_{k=K}^LT^k(I)=X$.
\end{corollary}
\begin{proof}
The first statement is a direct corollary of Theorem \ref{keane} and Lemma \ref{l4}. To see that the covering happens in finite time we need to use compactness.  Let $\overline{I}$ denote the closure of $I$ in  $\R$  and 
$\hat{I} := I \cup ((\overline X\setminus X) \cap \overline{I})$. Then $\hat{I}$ is open in the induced topology of $\hat{X} = \overline{X}$.  

Now consider $a \in \overline X\setminus X$. 
By the definition of IET any point
$b \in \operatorname{int}(X)$ close enough to $a$ satisfy that: the open interval $J$ whose endpoints are $b$ and $a$ is included in $X$, $a$ is in $\hat J$  and there is an open interval $J' \subset X\setminus\operatorname{Sing}(T)$ such that $T(J') =J$ (so $a$ is in $\hat{T(J')}=\hat{J}$).

By assumption
$\bigcup_{k \in \Z} {T^k (I)}$ covers $J'$ and $J$, thus $\bigcup_{k \in \Z} \hat{T^k (I)}$ covers $a$. Repeating this for all points $a\in\overline{X}\setminus X$, it shows that 
$\bigcup_{k \in \Z} \hat{T^k (I)}$ is a countable open cover of the compact set $\hat{X}$, and thus there exists $K,L$
such that $\bigcup_{k =K}^L \hat{T^k (I)}$ is a open cover of $\hat{X}$.
This immediately implies that $\bigcup_{k =K}^L {T^k (I)}$ is an open cover of $X$.
\end{proof}

\subsection{Application to rational polygonal billiards}
A {\em polygon} $P$ is a compact, finitely connected, planar domain whose boundary $\partial P$ consists of a finite union of segments.
We play billiards in $P$, take any point $s \in \partial P$ any $\theta \in \S1$ such that the vector $(s,\theta)$ points into
the interior of $P$; flow $(s,\theta)$ until it hits the boundary $\partial P$ and then reflect the direction with the usual law
of geometric optics, angle of incidence equals angle of reflection to produce the point $(s',\theta') =  T(s,\theta)$. $T$ is
called the billiard map, it is not defined if $s'$ is a corner of the polygon.  The inverse $T^{-1}$ is defined if $s$ is not a corner. 
A polygon is called {\em rational} if the angle between any pair of sides is a rational multiple of $\pi$.  Suppose that the
angles are $\pi\frac{m_i}{n_i}$ with $m_i$ and $n_i$ relatively prime; let $N$ be the least common multiple of the $n_i$, and
$D_N$ be the dihedral group generated by reflections in the lines through the origin that meet at angle $\frac{\pi}{N}$. 
Let $D_N(\theta)$ denote the $D_N$ orbit of a direction $\theta \in \S1$.
The following Theorem is a compilation of the well known results, see for example \cite{MaTa}[Sections 1.5~and~1.7]:

\begin{theorem} Suppose $P$ is a rational polygon.\\
i) For each $\theta \in \S1$ any orbit starting in  the direction $\theta$ only takes directions in the set  $D_N(\theta)$.\\
ii) For each $\theta \in \S1$ the restriction $T_\theta$ of the map $T$ to $\partial P \times D_N(\theta)$ is an interval exchange transformation in the sense defined in this appendix.\\
iii) If $\theta$ is non-exceptional, then $T_{\theta}$ is irreducible and has no connections.
\end{theorem}

Combining this theorem with Corollary \ref{c7} yields

\begin{corollary}\label{c9}
If $P$ is a rational polygon and $\theta$ is non-exceptional, then the map $T_\theta: X \to X$ is minimal, and thus the orbit
of every open interval $I$ covers $X$.
\end{corollary}

\end{document}